\documentclass[letterpaper,12pt]{amsart}
\title[Markov degree two]{The three-state toric homogeneous Markov chain model has Markov degree two}
\author{Patrik Nor\'en}
\address{Department of Mathematics and Systems Analysis, Aalto University}
\email{patrik.noren@aalto.fi}
\usepackage{latexsym,array,delarray,amsthm,amssymb,epsfig}
\usepackage{color}

\theoremstyle{plain}
\newtheorem{theorem}{Theorem}[section]
\newtheorem{lemma}[theorem]{Lemma}
\newtheorem{proposition}[theorem]{Proposition}

\newtheorem*{theoremNoNumber}{Theorem}
\newtheorem*{conjectureNoNumber}{Conjecture}

\theoremstyle{definition}
\newtheorem{definition}[theorem]{Definition}
\newtheorem{ex}[theorem]{Example}
\newtheorem{re}[theorem]{Remark}

\begin{document}
\begin{abstract}
We prove that the three-state toric homogenous Markov chain model has Markov degree two. In algebraic terminology this means, that a certain class of toric ideals are generated by quadratic binomials. This was conjectured by Haws, Martin del Campo, Takemura and Yoshida, who proved that they are generated by binomials of degree six or less.
\end{abstract}
\maketitle

\section{Introduction}
This paper considers the $S$-state toric homogeneous Markov chain model. Let $S$ and $T$ be positive integers, set
\[R_{S,T}=\mathbb{K}[x_w\mid\textrm{$w$ is a $T$-letter word $i_1\ldots i_T$ on the alphabet $[S]$ with $i_j\neq i_{j+1}$}],\]
and define the \emph{$S$-state toric homogeneous Markov $T$--chain ideal}, $I_{S,T}$, as the kernel
of the ring homomorphism
\[\Phi_{S,T}:R_{S,T}\rightarrow R_{S,2}\] 
given by $\Phi(x_{i_1\ldots i_T})=x_{i_1i_2}x_{i_2i_3}\cdots x_{i_{T-1}i_T}.$ The \emph{Markov degree} of an ideal is the smallest degree of a generating set, and the \emph{Gr\"obner degree} is the smallest degree of a Gr\"obner basis. The degree of a generating set or Gr\"obner basis is the highest degree of an element in the generating set or Gr\"obner basis.

\begin{theoremNoNumber}[Haws, Martin del Campo, Takemura and Yoshida, \cite{rudy1}, \cite{rudy2}]
The Markov degree of the three-state model ideal is at most six.
\end{theoremNoNumber}

\begin{conjectureNoNumber}[\cite{rudy1}, \cite{rudy2}]
For $S>2$, the $S$-state model ideal has Markov degree $S-1$ and Gr\"obner degree $S$.
\end{conjectureNoNumber}

We prove the Markov part of the conjecture for the three-state model ideal by combinatorial arguments.
\\ $\,$

\noindent {\bf Theorem 3.7.} \emph{The Markov degree of the three-state model ideal is two.} \\

The $S$-state toric homogeneous Markov $T$--chain ideal is similar to the ideal of graph homomorphisms from the path of length $T$ to the complete graph on $S$ vertices. The following result motivates the belief that the structure of $I_{S,T}$ should be possible to understand.

\begin{theoremNoNumber}[Engstr\"om and Nor\'en \cite{EN}]
The ideal of graph homomorphisms from any forest to any graph has a square-free quadratic Gr\"obner basis.
\end{theoremNoNumber}

This theorem was proved using the toric fiber product. In this paper, we will use an adaption of that object, similar to those in \cite{E}, \cite{EKS} and \cite{ss}, for ideals that are not always toric fiber products right off. The results of this paper are about algebraic statistics, an area further surveyed in the book by Drton, Sturmfels, and Sullivant~\cite{DSS}.

\begin{re}
If the variables of $I_{3,4}$ are ordered as  (1212, 1213, 1231, 1232, 1312, 1313, 1323, 2121, 2123, 2131, 2132, 2313, 2321, 2323, 3121, 3131, 3132, 3212, 3231, 3232) and a monomial order is given by the weight vector (7607, 8235, 1610, 1124, 7287, 9850, 3760, 9582, 4608, 4517, 2483, 5589, 4347, 2143, 161, 4947, 8497, 3128, 272, 3276) then the Gr\"obner basis from the weight vector is quadratic, which is sharper than conjectured. The order was found by computer calculations~\cite{ME}. The weight vectors giving quadratic Gr\"obner bases are rare and that explains why they were not found before.
\end{re}
\begin{re}
The general the conjecture for Markov degrees  is false, for example $I_{4,4}$ has Markov degree four. An indispensable generator of this ideal of degree four is $x_{1414}x_{2323}x_{4142}x_{4232}-x^2_{1423}x_{3232}x_{4141}$.
\end{re}
\section{State graphs and normal monomials}
Let $P_T$ be the directed path on vertex set $[T]$ with edges $12,23,\ldots,(T-1)T$ and let $K_3$ be the directed complete graph on vertex set $[3]$. Each $T$-letter word $i_1\ldots i_T$ on the alphabet $[3]$ with $i_j\neq i_{j+1}$ encodes a graph homomorphism $P_T\rightarrow  K_3$, by sending vertex $j$ to $i_j$. A \emph{state graph} is a directed graph on vertex set $[3]$ with multiple edges allowed but no loops.

The variables of $R_{3,T}$ are indexed by graph homomorphisms $P_T\rightarrow K_3$ each of which induces a state graph. If $x$ and $y$ are two variables with the same state graph, then $x-y\in I_{3,T}$. When describing the relations between higher degree monomials, it is enough to have one variable for each state graph.

The state graphs will be decomposed into paths and cycles. The notation for a path is $ij$ or $ijk$ depending on its length, the notation for a two-cycle is $(ij)$, and the notation for triangles is $(ijk)$ (the cycle $(ijk)$ have the edges $ij,jk,ki$.) We keep careful track of orientation: $12$, $23$, $31$, $123$, $231$, $312,$ and $(123)$ are oriented one way and $13$, $21$, $32$, $132$, $213$, $321$, and $(132)$ are oriented the other way.

\begin{proposition}\label{prop:basicgraph}
The state graph $G$ of a variable can be uniquely decomposed into a collection of two-cycles, triangles with the same orientation, and potentially a leftover path oriented in the same way as the triangles.
\end{proposition}
\begin{proof}
It is a basic fact in graph theory that a directed graph with the same in and out degree for each vertex can be decomposed into cycles.

The state graph comes from a graph homomorphism $P_T\rightarrow K_3$. If the homomorphism sends $1$ and $T$ to the same vertex, then the state graph can be decomposed into directed cycles. If there are triangles with different orientations, then any two oppositely oriented triangles can be replaced by a triple of two-cycles.

If the graph homomorphism sends $1$ to $i$ and $T$ to $j$ with $i\neq j$, then add an extra edge $ji$ in $G$ to get $G'$. Decompose $G'$ as before into cycles, with triangles oriented the same way. Then remove $ji$ from a triangle if possible, and otherwise from a two-cycle, to achieve compatible orientation.

That the decomposition is unique follows from the fact that all edges with one orientation are locked into two-cycles and the leftover edges are put into as many triangles as possible.
\end{proof}

\begin{proposition}\label{prop}
Any collection of two-cycles, triangles, and at most one path with the same orientation as the the triangles; with in total $T-1$ edges, is a decomposition of the state graph of a variable.
\end{proposition}
\begin{proof}
For each cycle $\alpha$, let $c_{\alpha}$ be the number of copies ofthe cycle $\alpha$ in the collection. The goal is to construct a word realizing the decomposition as a state graph of a variable. There are four different cases depending on what kind of path and cycles that occur in the decomposition. 

\textsc{Case 1.}
If the path $123$ occurs in the collection, then the word
\[\underbrace{1212\ldots12}_{c_{(12)}}\underbrace{ 1313\ldots13}_{c_{(13)}}\underbrace{123123\ldots 123}_{c_{(123)}}\underbrace{ 132132\ldots132} _{c_{(321)}}1\underbrace{2323\ldots23}_{c_{(23)}}23\]
realizes the collection as the state graph of a variable. The numbers under the brackets denote the number of times the small subword is repeated. For example, 
\[\underbrace{ 132132\ldots132} _{c_{(321)}}\]
represents $c_{(321)}$ copies of the word $132$.

\textsc{Case 2.}
If the path $12$ occurs, then the collection is realized by the word:
\[\underbrace{1212\ldots12}_{c_{(12)}}\underbrace{ 1313\ldots13}_{c_{(13)}}\underbrace{123123\ldots 123}_{c_{(123)}}\underbrace{ 132132\ldots132} _{c_{(321)}}1\underbrace{2323\ldots23}_{c_{(23)}}2.\]

\textsc{Case 3.}
If the collection only consists of cycles, and at least one of them is $(12)$, then the word
\[\underbrace{123123\ldots 123}_{c_{(123)}}\underbrace{ 132132\ldots132} _{c_{(321)}}\underbrace{ 1313\ldots13}_{c_{(13)}}1\underbrace{2323\ldots23}_{c_{(23)}}\underbrace{2121\ldots21}_{c_{(12)}}\]
gives a realization.

\textsc{Case 4.}
If the collection only consists of triangles, then the word
\[\underbrace{123123\ldots 123}_{c_{(123)}}\underbrace{ 132132\ldots132} _{c_{(321)}}1\]
gives a realization.

By symmetry, this proves that collections with $T-1$ edges come from words.
\end{proof}
\begin{ex}
The word $123231323123$ has the decomposition $(13)(23)(23)(123)123$.
\end{ex}
Since the state graph can be reconstructed from its decomposition, there is one variable associated to each decomposition. If $x,x',y,y'$ are variables from the decompositions $A,A',B,B'$ and $xy-x'y'\in I_{S,T}$, then we get the \emph{Markov step}

\[
\left[\begin{array}{c}A\\
B
\end{array}\right]\rightarrow\left[\begin{array}{c}A'\\
B'
\end{array}\right]
.\] The convention is that only the parts of the decomposition that is changed is written out, that is, the step

\[
\left[\begin{array}{c}AC\\
BD
\end{array}\right]\rightarrow\left[\begin{array}{c}AC'\\
BD'
\end{array}\right]
\] 
is written as
\[
\left[\begin{array}{c}C\\
D
\end{array}\right]\rightarrow\left[\begin{array}{c}C'\\
D'
\end{array}\right].
\]
After a step it might  be necessary to decompose the graphs in a new way. For example,
\[
\left[\begin{array}{c}(123)\\
(321)
\end{array}\right]\rightarrow\left[\begin{array}{c}(321)\\
(123)
\end{array}\right]
\]
could give
\[
\left[\begin{array}{c}(123)\\
(321)(321)
\end{array}\right]\rightarrow\left[\begin{array}{c}(321)\\
(12)(13)(23)
\end{array}\right].
\]
\begin{re}
The order of the cycles and paths in the decompositions does not matter. For example, $(12)(23)13$ is the same decomposition as $13(23)(12)$.
\end{re}

\begin{definition}
Let
\[
I^\le_{3,T}=\left\langle    b\in I_{3,T}\mid    \textrm{$b$ is a quadratic binomial whose Markov move changes at most $12$ edges}          \right\rangle.
\]

\end{definition}
The following normal form for monomials is useful.

\begin{definition}
A monomial $n$ is \emph{normal}, if
\begin{itemize}
\item[(1)]
all triangles in $n$ are oriented the same way;
\item[(2)]
and if two variables divide $n$, then the number of triangles in them differ by at most two;
\item[(3)]
and either
\begin{itemize}
\item[(a)]
all triangles and paths are oriented in the same way,
\item[(b)]
or there is at most one triangle in each variable in $n$, and there is no monomial $n'$ satisfying
\begin{itemize}
\item[(i)]
$n'-n\in I^\le_{3,T}$,
\item[(ii)]
and $n'$ has fewer triangles than $n$, or $n'$ has equally many triangles but fewer paths than $n,$
\end{itemize}
that is normal.
\end{itemize}
\end{itemize}
\end{definition}
\section{Proof of the main theorem}
To prove that the ideals $I_{3,T}$ are generated by quadrics, five lemmas are needed.
\begin{lemma}\label{lemma:main}
From a monomial $m$, it is possible by degree two-moves to reach a normal monomial $n$.
\end{lemma}
\begin{proof}
The first step is to get all the triangles of $m$ oriented the same way.

If the monomial $m$ has variables with triangles oriented differently, and a variable with more than one triangle, then the move \[
\left[\begin{array}{c}(123)\\
(321)
\end{array}\right]\rightarrow\left[\begin{array}{c}(321)\\
(123)
\end{array}\right]
\]decreases the number of triangles, since $(123)(321)$ becomes $(12)(13)(23)$ in the decomposition. After possible repetitions, either all triangles have the same orientation or the variables have at most one triangle. If there is a pair of variables with opposite oriented triangles, then there are two cases depending on if anyone of them have a path. If one of them have a path $P$, then it has the triangle, say $(123)$, with the same orientation. The step \[
\left[\begin{array}{c}(123)P\\
(321)
\end{array}\right]\rightarrow\left[\begin{array}{c}(321)P\\
(123)
\end{array}\right],
\]
reduce the number of triangles. If neither of them have a path, then the move \[
\left[\begin{array}{c}(123)\\
(321)
\end{array}\right]\rightarrow\left[\begin{array}{c}(12)31\\
(23)13
\end{array}\right]
\] reduces the number of triangles. Now all triangles can be assumed to have the same orientation, and Condition 1 is satisfied.

The second step is to reduce the difference between the number of triangles in the variables. If one variable contains at least three more triangles $T_1,$ $T_2,$ $T_3$ than another variable, then the other one contains at least three two-cycles $C_1,$ $C_2$ and $C_3$. The move
\[
\left[\begin{array}{c}T_1T_2\\
C_1C_2C_3
\end{array}\right]\rightarrow\left[\begin{array}{c}C_1C_2C_3\\
T_1T_2
\end{array}\right]
\]
reduces the difference. After repetitions, Condition 2 is satisfied.

If all paths and triangles have the same orientation, then the monomial is normal, since Condition 3.a is satisfied.

To show that Condition 3.b is satisfied we first find moves to a monomial with at most one triangle in each variable.

If there are no triangles, then we are done. Otherwise,  there is a triangle and a path $P$ with opposite orientations. The path $P$ is not in a variable with  a triangle, and thus no variable contains more than two triangles due to that Condition 2 is satisfied. If there is a variable with two triangles and a path $Q$, then $P$ and $Q$ are equally long, and the step
\[
\left[\begin{array}{c}P\\
Q
\end{array}\right]\rightarrow\left[\begin{array}{c}Q\\
P
\end{array}\right]
\]
reduce the number of triangles. From now on, we assume that no variables with two triangles have a path. If there are no variables with two triangles and no paths, then all variables have at most one triangle. Now assume that there is a variable with two triangles, say $(123)$, and no path. By parity the path $P$ have two edges, and is of opposite orientation, say $321$. The move 
\[
\left[\begin{array}{c}321\\
(123)
\end{array}\right]\rightarrow\left[\begin{array}{c}(12)\\
(23)31
\end{array}\right]
\]
reduces the number of triangles. This procedure can be repeated as long as there are variables with more than one triangle, and there are triangles and paths of opposite orientation. Thus, we either get everything oriented in the same way, and satisfy Condition 3.a, or get at most one triangle in each variable. Using quadratic Markov moves changing at most $12$ edges, minimize according to satisfy Condition 3.b.
\end{proof}
\begin{definition}
The \emph{$ij$-spin} of a variable $x$ is $s_{ij}(x)=c_{ij}-c_{ji}$ where $c_{kl}$ is the number of $kl$ edges in the state graph of $x$. The \emph{$ij$-spin} of a monomial $m=x_1\cdots x_d$ is $s_{ij}(m)=\Sigma_{k=1}^ds_{ij}(x_k)$. The \emph{spin-vector} of a monomial $m$ is
the vector $(s_{12}(m),s_{31}(m),s_{23}(m))$. The \emph{total spin} of a monomial $m$ is $s_{12}(m)+s_{23}(m)+s_{31}(m)$.
\end{definition}
Although there are homotopy and valuation theoretic interpretations of spin, we only use it for combinatorial calculations.
\begin{lemma}\label{lemma:direct}
Let $m$ and $n$ be normal monomials with $m-n\in I_{3,T}$. If $m$ has all paths and triangles oriented the same way, then so does $n$.
\end{lemma}
\begin{proof}
The proof is divided into four cases: Whether or not $n$ has triangles in any of the variables and parity of $T$. By symmetry, we can assume $m$ has orientation $(123)$. For contradiction, assume that $n$ has a path or triangle with orientation $(321)$.

\textsc{Case A.}
There are no triangles in $n$.

\textsc{Case A.1.}
Let $T$ be even.

By assumption, $n$ contains a path with $(321)$ orientation. This path has an odd number of edges as there is an odd number of edges in the state graph of each variable. Every variable in $m$ has a triangle or a path. In $m$, everything has the same orientation, so the total spin is at least the degree of $m$. In $n$, all variables contribute $1$ or $-1$ to the total spin, and at least one of the variables contributes $-1$. Thus, $n$ has total spin strictly less than its degree, contradicting that $m$ and $n$ have the same total spin.

\textsc{Case A.2.}
Let $T$ be odd.

By assumption, $n$ contains a path with $(321)$ orientation. This path has an even number of edges as there is an even number of edges in the state graph of the variable. By symmetry, we can assume $n$ contains the path $321$. Now, $n$ cannot contain the path $123$ since $n$ is normal and that would allow a move reducing the number of paths:
\[
\left[\begin{array}{c}321\\
123
\end{array}\right]\rightarrow\left[\begin{array}{c}(12)\\
(23)
\end{array}\right]
.\]

 To cancel the negative $23$-spin from a $321$ path, $n$ must contain a $23$ edge outside a two-cycle. That edge must be in a path since there are no triangles in $n$. The only path of the right length and orientation is $231$ since $123$ is excluded. The same argument for $12$ gives that $n$ contains a $312$ path for every path $321$. Since $n$ is normal, the only type of path oriented as $(321)$ in $n$ is $321$. All spin in $n$ is from the paths $312,231,312$, so the $31$-spin is strictly greater than the sum of the $12$-spin and $23$-spin in $n$. The only way to get a similar contribution to $31$-spin in $m$ is from $31$ paths without triangles, but that is impossible since $m$ is normal and $T$ is odd.

\textsc{Case B.}
There are triangles in $n$.

\textsc{Case B.1.}
Let $T$ be even.

If all triangles in $n$ have orientation $(321),$ then there are variables with paths $12,23$ and $31$ in $n$. If there is a variable in $n$ with a triangle and no path, then the move
\[
\left[\begin{array}{c}(321)\\
12
\end{array}\right]\rightarrow\left[\begin{array}{c}(12)32\\
13
\end{array}\right]
\]
reduces the number of triangles in $n$, contradicting normality. If there is a path on the variable with a triangle, then moves of the type
\[
\left[\begin{array}{c}321\\
12
\end{array}\right]\rightarrow\left[\begin{array}{c}(12)\\
32
\end{array}\right]
\]
reduce the number of paths. 

Next, we consider the case that all triangles in $n$ have orientation $(123)$. By assumption, there is a variable in $n$ with a $(321)$ oriented path. This variable does not contain a triangle since it is oriented differently and $n$ is normal. This is a single edge path since $T$ is even. By symmetry, let the path be $21$. If $n$ contains the path $123$, then the move
\[
\left[\begin{array}{c}21\\
123
\end{array}\right]\rightarrow\left[\begin{array}{c}23\\
(12)
\end{array}\right]
\]
reduces the number of paths, contradicting normality. If $n$ contains a variable with a triangle and no path, then the move 
\[
\left[\begin{array}{c}21\\
(123)
\end{array}\right]\rightarrow\left[\begin{array}{c}23\\
(12)31
\end{array}\right]
\]
reduces the number of triangles, again contradicting normality.
The variables in $n$ with triangles have exactly one triangle since $n$ is normal. Furthermore, they have a path on two edges since $T$ is even. The only such path with correct orientation is $231$. The path $12$ cannot occur in $n$ since the move
\[
\left[\begin{array}{c}12\\
231
\end{array}\right]\rightarrow\left[\begin{array}{c}31\\
123
\end{array}\right]
\]
creates a variable with the path $123$, yielding a contradiction as earlier. If $n$ contain a $32$ path, then the move
\[
\left[\begin{array}{c}32\\
231
\end{array}\right]\rightarrow\left[\begin{array}{c}31\\
(23)
\end{array}\right]
\]
reduces the number of paths, contradicting normality. Likewise, the move
\[
\left[\begin{array}{c}13\\
231
\end{array}\right]\rightarrow\left[\begin{array}{c}23\\
(13)
\end{array}\right]
\]
contradicts that $n$ contains the path $13.$ The set of spin vectors \newline$(s_{12}(x),s_{31}(x),s_{23}(x))$ of variables $x$ potentially occurring in $n$ is:
\[\{(-1,0,0),(0,1,0),(0,0,1),(1,2,2)\},\]
and $(-1,0,0)$ occurs. The spins then satisfy $s_{23}(n)+s_{31}(n)-3s_{12}(n)-d>0$ where $d$ is the degree of $n$. The $d$ variables $x$ in $m$ all satisfy $s_{23}(x)+s_{31}(x)-3s_{12}(x)-1\le0$ since $m$ is normal and oriented $(123).$

\textsc{Case B.2.}
Let $T$ be odd.

First, we consider the case that all triangles are oriented $(321)$. Every variable with triangles has at most one triangle since $n$ is normal. With every triangle comes a single edge path since $T$ is odd. By symmetry, we assume that $13$ is a path in $n$. To get a non-negative $31$-spin, the edge $13$ is compensated by a path with $31$. This path is not in a variable with a triangle since they would have different orientation. The path with $31$ contains two edges since $T$ is odd. There are two options: $231$ and $312$.

However, the moves
\[
\left[\begin{array}{c}(321)13\\
312
\end{array}\right]\rightarrow\left[\begin{array}{c}(12)(13)\\
132
\end{array}\right]\textrm{ and }
\left[\begin{array}{c}(321)13\\
231
\end{array}\right]\rightarrow\left[\begin{array}{c}(12)(23)\\
213
\end{array}\right]
\]
reduce the number of triangles, contradicting that $n$ is normal.

Now, we consider the case of all triangles oriented $(123)$. By assumption, there is a path in $n$ with the orientation $(321)$ and this variable has no triangle since $n$ is normal. The path has two edges since $T$ is odd. By symmetry, we assume that $n$ contains the path $321$. All variables with triangles in $n$ have a path with one edge since $n$ is normal. If that path is not $31$, then the moves
\[
\left[\begin{array}{c}(123)12\\
321
\end{array}\right]\rightarrow\left[\begin{array}{c}(12)312\\
(23)
\end{array}\right]
\textrm{ or }
\left[\begin{array}{c}(123)23\\
321
\end{array}\right]\rightarrow\left[\begin{array}{c}(23)231\\
(12)
\end{array}\right]
\]
decrease the number of triangles, contradicting normality. Thus, all paths with the $(321)$ orientation have to be $321$ and all variables with triangles have the path $31$. The set of spin vectors $(s_{12}(x),s_{31}(x),s_{23}(x))$ of variables $x$ potentially occurring in $n$ is
\[\{(0,0,0),(0,1,1),(-1,0,-1),(1,1,0),(1,2,1)\},\]
 and $(-1,0,-1)$ occurs. The spins satisfy $s_{12}(n)-s_{13}(n)+s_{23}(n)<0$ while the variables $x$ in $m$ all satisfy $s_{12}(x)-s_{13}(x)+s_{23}(x)\ge 0$.
\end{proof}

\begin{lemma}\label{lemma:diff}
Let $m$ and $n$ be normal monomials with $m-n\in I_{3,T}$. The maximal number of triangles in a variable in $m$ cannot be less than the minimal number of triangles in a variable in $n$.
\end{lemma}
\begin{proof}
If $m$ has variables with different orientation, then so does $n$ by Lemma~\ref{lemma:direct}. In this case, both monomials have at most one triangle in each variable and both monomials have variables with no triangles.

Now, we consider the case that both monomials have all variables oriented the same way. The total spin is different for $m$ and $n$ if the maximal number of triangles in a variable in $m$ is less than the minimal number of triangles in a variable in $n$.
\end{proof}

\begin{lemma}\label{lemma:directed}
If $m$ and $n$ are normal monomials such that $m-n\in I_{3,T}$ and $m$ and $n$ have all paths and triangles oriented in the same way, then it is possible to use degree two-steps to go from $m$ to $n.$
\end{lemma}
\begin{proof}
By symmetry, we can assume that the orientation of the monomials is $(123).$ The proof is structured as follows: We start with two monomials $m$ and $n$ of degree $d$ and no common variables. Then, a sequence of steps is presented from $m$ to $m'$ and from $n$ to $n'$ so that they share a variable. By induction on $d$, it is then possible to go between $m$ and $n$ by degree two steps. The base case is trivial. The induction step is split into two cases: A. There is a variable $x$ in $m$ and a variable $y$ in $n$ that have the same number of triangles; B. Otherwise.

\textsc{Case A.}
There is a variable $x$ in $m$ and variable $y$ in $n$ that have the same number of triangles.

By parity, there can not be paths in $x$ and $y$ of different lengths. This case is split into four different subcases: 1. The variables $x$ and $y$ have the same path; 2. The variables $x$ and $y$ have different one edge paths; 3. The variables $x$ and $y$ have different two edge paths; 4. The variable $x$ has no path and $y$ has a path with two edges.

\textsc{Case A.1.}
The variables $x$ and $y$ have the same path.

The state graphs of $m$ and $n$ are the same. All edges with orientation $(321)$ are in two-cycles since all variables in $m$ and $n$ have the orientation $(123)$. All two cycles contain an edge with orientation $(321)$. This proves that the collection of two cycles in $m$ is the same as the collection of two-cycles in $n$. Now it is possible to do a sequence of moves 
\[
\left[\begin{array}{c}(i_1i_2)\\
(i_3i_4)
\end{array}\right]\rightarrow\left[\begin{array}{c}(i_3i_4)\\
(i_1i_2)
\end{array}\right]
\]
to $m',n'$ that share a variable since any subcollection of two-cycles from $m$ and $n$ can be picked and moved to these variables.

\textsc{Case A.2.}
The variables $x$ and $y$ have different paths with one edge.

By symmetry, let the path on $x$ be $12$ and let the edge on $y$ be $23$. If the monomial $m$ contains the paths $23$ or $231$, then the moves
\[
\left[\begin{array}{c}12\\
23
\end{array}\right]\rightarrow\left[\begin{array}{c}23\\
12
\end{array}\right]\textrm{ or }\left[\begin{array}{c}12\\
231
\end{array}\right]\rightarrow\left[\begin{array}{c}23\\
312
\end{array}\right]
\]
create a monomial $m'$ that contains a variable $x'$ with the same number of edges as $y$ and the same path. This reduces to Case 1.

If $m$ does not contain $23$ or $231$, then any edge $23$ not in a two-cycle or a triangle is in a path $123$. In particular, the $12$-spin of $m$ and $n$ is strictly greater than the $23$-spin of $m$ and $n$. By a similar argument for $n$, the steps
\[
\left[\begin{array}{c}23\\
12
\end{array}\right]\rightarrow\left[\begin{array}{c}12\\
23
\end{array}\right]\textrm{ or }\left[\begin{array}{c}23\\
312
\end{array}\right]\rightarrow\left[\begin{array}{c}12\\
231
\end{array}\right],
\]
give either a reduction to Case 1 or that the $23$-spin of $n$ and $m$ is greater than the $12$-spin of $n$ and $m$, a contradiction.

\textsc{Case A.3.}
The variables $x$ and $y$ have different two edge paths.

By symmetry, let $x$ have the path $123$ and let $y$ have the path $231$. If $m$ contains the path $231$ or $n$ contains the path $123$, then Case 1 applies after a swap of paths. If both $m$ and $n$ contain the path $312$, then, after the moves 

\[
\left[\begin{array}{c}123\\
312
\end{array}\right]\rightarrow\left[\begin{array}{c}312\\
123
\end{array}\right]\textrm{ and }
\left[\begin{array}{c}231\\
312
\end{array}\right]\rightarrow\left[\begin{array}{c}312\\
231
\end{array}\right],
\]
Case 1 applies. If $n$ contains a path $12$ or $m$ contains the path $31$,  then the steps
\[
\left[\begin{array}{c}231\\
12
\end{array}\right]\rightarrow\left[\begin{array}{c}123\\
31
\end{array}\right]
\textrm{ and }
\left[\begin{array}{c}123\\
31
\end{array}\right]\rightarrow\left[\begin{array}{c}231\\
12
\end{array}\right]
\]
reduce to Case 1.

Neither $n$ nor $m$ contain $312$. If $n$ contains no $312$-path, then the $31$-spin of $n$ and $m$ is strictly greater than the $12$-spin of $n$ and $m$ while in $m$, all $31$-edges not in two-cycles are in triangles or paths containing $12$. This is a contradiction.

Similarly, if  $m$ contains no $312$,  then the $12$-spin of $m$ and $n$ is strictly greater than the $31$-spin of $m$ and $n$ while in $n$, all $12$-edges not in two-cycles are in triangles or paths containing $31$. This is also a contradiction.

\textsc{Case A.4.}
The variable $x$ has no path and $y$ has a path with two edges.

By symmetry, we can assume that $y$ has the path $123$. If $m$ contains any variable with a path on two edges, then swapping that path to $x$ with a two-cycle gives a reduction to Case 1 or 3. If $m$ contains a variable with a triangle and a path with one edge, then moves of type
\[
\left[\begin{array}{c}(i_1i_2)\\
(ijk)ij
\end{array}\right]\rightarrow\left[\begin{array}{c}ijk\\
(i_1i_2)kij
\end{array}\right]
\]
give a reduction to Case 1 or 3.

If $m$ contains a variable other than $x$ that has a triangle with no path, then the move
\[
\left[\begin{array}{c}(i_1i_2)\\
(ijk)
\end{array}\right]\rightarrow\left[\begin{array}{c}ijk\\
(i_1i_2)ki
\end{array}\right]
\]
gives a reduction to Case 1 or 3.

The remaining case is that all variables in $m$ except $x$ contain no triangles and no paths of length two. If $T$ is odd, then the other variables contain no paths and if $T$ is even, then the other variables have paths with one edge. The variable $x$ contains fewer edges outside two-cycles than $y$ and the other variables in $m$ contain the lowest possible number of edges outside two-cycles. Thus, the total spin of $m$ and $n$ cannot be the same, a contradiction.

\textsc{Case B.}
There are no variables in $m$ and $n$ that have the same number of variables.

By Lemma~\ref{lemma:diff} and symmetry, we have that for some integer $t$: The variables in $n$ have $t$ or $t+2$ triangles; the variables in $m$ have $t-1$ or $t+1$ triangles; there are variables with $t,t+1$, and $t+2$ triangles. By parity, there are two subcases: 1. The variables in $m$ have paths with one edge; 2. The variables in $y$ have paths with one edge.

\textsc{Case B.1.}
The variables in $m$ have paths with one edge.

If all paths on the variables are the same, then this edge would get $d$ higher spin than any other edge. This is a contradiction because the paths in $n$ all have two edges and the edge spins are more evenly distributed. Thus, there are different paths in $m$. By symmetry, we can assume the paths are $12$ and $23$. The move
\[
\left[\begin{array}{c}(123)12\\
(i_1i_2)23
\end{array}\right]\rightarrow\left[\begin{array}{c}(i_1i_2)123\\
(123)
\end{array}\right]
\]
almost gives a reduction to Case A. It is possible that the new monomial needs to be normalized first by swapping two triangles for three two-cycles, but then Case A applies. If $12$ is not on a variable with a triangle to start off with, we can swap it that way.

\textsc{Case B.2.}
The variables in $n$ have paths with one edge.

The same type of argument as in Case 1 applies, with $m$ and $n$ switched.
\end{proof}

\begin{lemma}\label{lemma:undirected}
If $m$ and $n$ are normal monomials, both with paths and triangles in different orientations, and if $m-n\in I_{3,T}$, then it is possible to go from $m$ to $n$ using degree two steps.
\end{lemma}
\begin{proof}
Divide into four cases: By parity of $T$, and whether or not $m$ has triangles in any of the variables. 

\textsc{Case A.}
Let $T$ be even.

\textsc{Case A.1.}
There are triangles in $m$.

By symmetry, let the triangles in $m$ have orientation $(123)$. By assumption, $m$ have a path of orientation $(321)$, and by parity and symmetry one of these paths is $21$. If a triangle has the paths $123$ or $312$, then the moves
\[
\left[\begin{array}{c}123\\
21
\end{array}\right]\rightarrow\left[\begin{array}{c}(12)\\
23
\end{array}\right]\textrm{ or }\left[\begin{array}{c}312\\
21
\end{array}\right]\rightarrow\left[\begin{array}{c}(12)\\
31
\end{array}\right]
\]
reduce the number of paths, contradicting normality. If there is a triangle with no path, then the move
\[
\left[\begin{array}{c}(123)\\
21
\end{array}\right]\rightarrow\left[\begin{array}{c}(12)23\\
31
\end{array}\right]
\]
reducse the number of triangles, again contradicting normality. Thus, every triangle in $m$ has the path $231$. We assumed that there is a triangle in $m$, so the path $231$ is in $m$. Using that path and moves similar to those above, we exclude the existence of paths $13$ and $32$. Furthermore, if the path $12$ exists, then the move
\[
\left[\begin{array}{c}231\\
12
\end{array}\right]\rightarrow\left[\begin{array}{c}123\\
31
\end{array}\right]
\]
leads to a contradiction, as one of the earlier moves demonstrated.

The only paths left are $21,23,31,$ and $231$. Thus, the only possible spin vectors of variables in $m$ are $(-1,0,0),(0,1,0),(0,0,1),$ and $(1,2,2)$ with $(1,2,2)$ and $(-1,0,0)$ occurring.

If $n$ has triangles, then the same argument gives six possible sets of spin vectors, due to symmetry break. However, only one of the six have the two last coordinates positive and larger than the first coordinate and that is the same as for $m$. The variables $m$ and $n$ contain the same types of paths and triangles. To match up the spins, they have to be equally many of each type. It suffices to swap two-cycles using degree two steps.

If $n$  does not  contain triangles, then $|s_{12}(n)|+|s_{23}(n)|+|s_{31}(n)|$ is at most the degree of $n$. From the list of paths in $m$, we know that the same expression for $m$ is larger than the degree of $m$, a contradiction.

\textsc{Case A.2.}
There are no triangles in $m$.

If there are triangles in $n$, then this is Case 1, so assume that $n$ does not contain triangles. Using the steps
\[\left[\begin{array}{c}(i_1i_2)k_1k_2\\
(j_1j_2)k_2k_1
\end{array}\right]\rightarrow\left[\begin{array}{c}(i_1i_2)j_1j_2\\
(k_1k_2)j_2j_1
\end{array}\right]\textrm{ and }\left[\begin{array}{c}(ij)\\
(k\ell)
\end{array}\right]\rightarrow\left[\begin{array}{c}(k\ell)\\
(ij)
\end{array}\right]
\]
we get from $m$ to $n$ by degree two moves.

\textsc{Case B.}
Let $T$ be odd.

\textsc{Case B.1.}
There are triangles in $m$.

By symmetry, the triangles have the orientation (123) and $m$ contains the path $321$. If $m$ contain a variable with the path $12$ or $23$, then the moves 
\[
\left[\begin{array}{c}321\\
(123)23
\end{array}\right]\rightarrow\left[\begin{array}{c}231\\
(12)(23)
\end{array}\right]\textrm{ or }
\left[\begin{array}{c}321\\
(123)12
\end{array}\right]\rightarrow\left[\begin{array}{c}312\\
(12)(23)
\end{array}\right]
\]
reduce the number of triangles, contradicting normality. All paths on variables with triangles have to be $31$. If the paths $132$ or $213$ occur, then the moves
\[
\left[\begin{array}{c}132\\
31
\end{array}\right]\rightarrow\left[\begin{array}{c}(13)\\
32
\end{array}\right]\textrm{ or }
\left[\begin{array}{c}213\\
31
\end{array}\right]\rightarrow\left[\begin{array}{c}(13)\\
21
\end{array}\right]
\]
reduce the number of paths, contradicting normality. All paths with orientation $(321)$ are $321$. If the path $123$ occurs, then the move
\[\left[\begin{array}{c}123\\
321
\end{array}\right]\rightarrow\left[\begin{array}{c}(12)\\
(23)
\end{array}\right]
\]
reduces the number of paths and contradicts normality. If there is a variable with no triangle and no path, then the move
\[\left[\begin{array}{c}(123)31\\
(ij)
\end{array}\right]\rightarrow\left[\begin{array}{c}(ij)312\\
231
\end{array}\right]
\]
reduces the number of triangles, contradicting normality. The paths that occur are $321,31,$ $231$ and $312$. Now proceed exactly as in Case A.1.
 
\textsc{Case B.2.}
There are no triangles in $m$.

If there are triangles in $n$, then this is Case 1, so assume that there are no triangles in $n$.

The steps
\[\left[\begin{array}{c}ijk\\
kji
\end{array}\right]\rightarrow\left[\begin{array}{c}(ij)\\
(jk)
\end{array}\right]
\]
give that the monomials cannot contain both orientations of a path.
If both the paths $123$ and $231$ are more common in $m$ than in $n$, then the $23$-spin cannot be equal for both monomials. Thus, there is at most one type of path in any orientation that is more common in $m$ than in $n$. By symmetry, there is one path in each orientation for which the monomials $m$ and $n$ have an equal number. These paths have one undirected edge $ij$ in common. The $ij$-spin gives that the other paths containing directed $ij$ are equally common. Then the total spin gives that both monomials $m$ and $n$ have the same number of each path. The collection of paths and two-cycles are the same. In this situation swapping two-cycles is enough to go between $m$ and $n$.
\end{proof}
\begin{theorem}
The three-state toric Markov chain model is generated by quadrics.
\end{theorem}
\begin{proof}
Let $m',n'$ be two monomials with $m'-n'\in I_{3,T}$ and let $m,n$ be the corresponding normal monomials from Lemma~\ref{lemma:main}.
According to Lemma~\ref{lemma:direct}, either both $m$ and $n$ have everything oriented in the same way or they have paths and triangles oriented differently.
 It is possible to go between $m$ and $n$ by Lemma~\ref{lemma:directed} and Lemma~\ref{lemma:undirected}.
\end{proof}
\begin{re}
In the Markov moves introduced, at most $12$ edges were interchanged. This shows that the Markov moves, and the generating binomials, fall into a finite number of symmetry classes.
\end{re}
\noindent {\bf Acknowledgements.} I would like to thank Jan Draisma and Ruriko Yoshida for their comments and questions.

\end{document}